\documentclass[amsmath,amscd,amsbsy,amssymb]{amsart}
 \relax

\chardef\bslash=`\\ % p.  424, TeXbook %\newcommand{\ntt}{\seriesm\shape n\tt}

\makeatletter \def\verbatim{\interlinepenalty\@M \@verbatim
\leftskip\@totalleftmargin\advance\leftskip2pc \frenchspacing\@vobeyspaces
\@xverbatim} \makeatother \hfuzz1pc
\makeatletter \def\dgt@k{\dg@DX=-3 \dg@DY=2 \dg@SIZE=3} \makeatother
\makeatletter \def\dgt@kk{\dg@DX=3 \dg@DY=-1 \dg@SIZE=3}% \makeatother
\theoremstyle{plain} \newtheorem{thm}{Theorem}[section]
\newtheorem{cor}[thm]{Corollary} \newtheorem{lemma}[thm]{Lemma}
\newtheorem{prop}[thm]{Proposition}
\theoremstyle{definition} 
 \newtheorem{que}[thm]{Question}
\newcommand{\A}{\mathcal A}
\newcommand{\LL}{\mathcal L}

\numberwithin{equation}{section}
\newcommand{\aed}{{\mathrm a}{\mathrm s} {\mathrm -}{\mathrm e}{\mathrm x}
{\mathrm t} {\mathrm -}{\mathrm d}{\mathrm i}{\mathrm m}}
\newcommand{\ed}{{\mathrm e}{\mathrm x}
{\mathrm t}{\text -}{\mathrm d}{\mathrm i}{\mathrm m}}
\newcommand{\asd}{{\mathrm a}{\mathrm s} {\mathrm d}{\mathrm i}{\mathrm m}}
\newcommand{\Var}{{\mathrm V}{\mathrm a}{\mathrm r}}
\newcommand{\R}{{\mathbb R}}
\newcommand{\OO}{{\mathcal O}}
\newcommand{\N}{{\mathbb N}}

\newcommand{\diam}{{\mathrm d}{\mathrm i} {\mathrm a}{\mathrm m}\ }

\usepackage[all]{xy} %\usepackage[ukrainian]{cyrtext}
\begin{document}

%%%%%%% Begin Topmatter %%%%%%%%%%
\title[On asymptotic extension dimension]{On asymptotic extension dimension}

\author{Du\v{s}an Repov\v{s}}
\address{Faculty of Mathematics and Physics, and Faculty of Education, 
University of Ljubljana,
Jadranska 19, 1000 Ljubljana, Slovenia} \email{dusan.repovs@guest.arnes.si}

\author{Mykhailo Zarichnyi}
\address{Department of Mechanics and Mathematics, Lviv National University,
Universytetska 1, 79000 Lviv, Ukraine} \email{mzar@litech.lviv.ua}



\begin{abstract}{The aim of this paper  is to introduce an asymptotic
counterpart of the extension dimension defined by Dranishnikov. The
main result establishes a relation between the asymptotic
extensional dimension of a proper metric space and extension
dimension of its Higson corona.}
\end{abstract}

\date{\today}

\thanks{This research was supported by the
Slovenian Research Agency grants P1-0292-0101 and J1-2057-0101.}

\maketitle

\section{Introduction}\label{S:intro}

Asymptotic dimension defined by Gromov \cite{G} has been
an object of
study in numerous publications (see expository paper
\cite{BD1}).
A metric space $(X,d)$ is of asymptotic dimension $\le n$ (written
$\mathrm{asdim}\,X\le n$) if for every $D>0$ there exists a
uniformly bounded cover $\mathcal U$ of $X$ such that $\mathcal
U=\mathcal U^0\cup\dots\cup\mathcal U^n$, where every family
$\mathcal U^i$ is $D$-disjoint, $i=0,1,\dots,n$. Recall that a
family $\mathcal A$ of subsets of $X$ is {\em uniformly bounded}
if $$\mathrm{mesh}\,\mathcal A=\sup\{\mathrm{diam}\,A\mid
A\in\mathcal A\}<\infty$$ and is called $D$-{\em disjoint} if
$\inf\{d(a,a')\mid a\in A,\ a'\in A'\}>D$, for every distinct
$A,A'\in \mathcal A$.

The asymptotic dimension can be characterized in different terms;
in particular, in terms of extension of maps into euclidean spaces.
The aim of this paper is to introduce an asymptotic analogue of
the extension dimension introduced by Dranishnikov in \cite{D2,
D}.

\section{Preliminaries}\label{S:prelim}

A typical metric is denoted by $d$. By $N_r(x)$ we denote the open ball of
radius $r$ centered at a point $x$ of a metric space.

\subsection{Asymptotic category}
 A map $f\colon X\to Y$ between metric
spaces is called $(\lambda,\varepsilon)$-{\em Lipschitz\/} for $\lambda>0$,
$\varepsilon\ge0$ if $d(f(x),f(x'))\le\lambda d(x,x')+\varepsilon$ for every
$x,x'\in X$. A map is called {\em asymptotically Lipschitz\/} if it is
$(\lambda,\varepsilon)$-Lipschitz for some $\lambda,\varepsilon>0$.
The $(\lambda,0)$-Lipschitz maps are also called {\em $\lambda$-Lipschitz\/},
$(1,0)$-Lipschitz maps are also called {\em short}.
A metric space $X$ is called {\em proper\/} if every closed ball in $X$
is compact.

The {\em asymptotic category\/} $\A$ was introduced by Dranishnikov \cite{D}.
The objects of $\A$ are proper metric spaces and the morphisms are proper
asymptotically Lipschitz maps. Recall that a map is called {\em proper} if the
preimage of every compact set is compact.

We also need the notion of a {\em coarse map}. A map between proper metric
spaces is called {\em coarse uniform\/} if
for
every $C>0$ there is $K>0$ such that for every $x,x'\in X$ with $d(x,x')<C$ we
have $d(f(x),f(x'))<K$. A map $f\colon X\to Y$  is called {\em metric proper\/}
if the preimage $f^{-1}(B)$ is bounded for every bounded set $B\subset Y$.
A map is {\em coarse} if it is metric proper and coarse uniform.

\subsection{Higson compactification and Higson corona}

Let $\varphi\colon X\to\R$ be a function defined on a metric space $X$.
For every $x\in X$ and every $r>0$ let
$V_r\varphi(x)=\sup\{|\varphi(y)-\varphi(x)|\mid
y\in N_r(x)\}$. A function $\varphi$ is called {\it slowly oscillating\/}
whenever for every $r>0$ we have
$V_r\varphi (x)\to0$ as $x\to\infty$ (the latter means that for every
$\varepsilon>0$
there exists a compact subspace $K\subset X$ such that
$|V_r\varphi(x)|<\varepsilon$
for all $x\in X\setminus K$. 

Let $\bar X$ be the compactification of $X$ that
corresponds
to the family of all continuous bounded slowly oscillation functions. The {\it
Higson
corona\/} of $X$ is the remainder $\nu X=\bar X\setminus X$ of this
compactification.
It is well-known that the Higson corona is a functor from the category of proper
metric space and coarse maps into the category of compact Hausdorff spaces. In
particular, if $X\subset Y$, then $\nu X\subset \nu Y$.

For any subset $A$ of $X$ we denote by $A'$ its trace on $\nu X$, i. e. the
intersection of the closure of $A$ in $\bar X$ with $\nu X$. Obviously, the set
$A'$ coincides with the Higson corona $\nu A$.

\subsection{Cones}

Let $X$ be a metric space of diameter $\le1$. The {\em open cone\/}
of $X$ is the set $\OO X=(X\times\R_+)/(X\times\{0\})$ endowed with
the metric (by $[x,t]$ we denote the equivalence class of $(x,t)\in
X\times\R_+$):
$$d([x_1,t_1],[x_2,t_2])=|t_1-t_2|+\min\{t_1,t_2\}d(x_1,x_2).$$
For a map $f\colon X\to Y$ of metric spaces we denote by $\OO f\colon \OO X\to
\OO Y$ the
map
defined as $\OO f([x,t])=[f(x),t]$.

\begin{prop} If $f\colon X\to Y$ is a Lipschitz map than $\OO f$ is an
asymptotically
Lipschitz map.
\end{prop}
\begin{proof} Suppose a map $f\colon X\to Y$ is $\lambda$-Lipschitz. Then for
any $[x_1,t_1], [x_2,t_2]\in \OO X$ we have
\begin{align*}d(\OO f([x_1,t_1]), \OO f([x_2,t_2]))=&d([f(x_1),t_1],
[f(x_2),t_2])\\=&|t_1-t_2|+\min\{t_1,t_2\}d(f(x_1),f(x_2))\\ \le& \lambda'
(|t_1-t_2|+\min\{t_1,t_2\}d(x_1,x_2)),
\end{align*}
where $\lambda'=\max\{\lambda,1\}$.
\end{proof}
The open cone of a finite CW-complex is a coarse CW-complex in the sense of
\cite{M}.

 Denote by $\alpha_L\colon \OO L\to\R$ the function defined by
$\alpha_L([x,t])=t$.
Obviously, $\alpha_L$ is a short function.
Let $\tilde \OO L=\{[x,t]\in \OO L\mid t\ge1\}$. Denote by $\beta_L\colon \tilde
\OO L\to
L$ the map $\beta_L([x,t])=x$.
\begin{lemma} The map $\beta_L$ is slowly oscillating.
\end{lemma}
\begin{proof} For $R>0$, the $R$-ball centered at $[x,0]$ is $\{[x,t]\mid t<R$.
If $d([x,t],[x_1,t_1])<K<R$, then $|t-t_1|+\min\{t,t_1\}d(x,x_1)<K$, i.e.
$(t-R)d(x,x_1)<R$ and $d(x,x_1)<K/(t-K)$. Therefore,
$d(\beta_L(x),\beta_L(x_1))<K/(R-K)\to 0$ as $R\to\infty$.
\end{proof}

Let $\bar\beta_L\colon \tilde \OO L\to L$ be the (unique) extension of the map
$\beta_L$. Denote by $\eta_L\colon \nu\tilde \OO L\to L$ the restriction of
$\beta_L$.
\begin{prop} Let $f\colon A \to \OO L$ be a proper asymptotically Lipschitz map
defined on a proper closed subset $A$ of a proper metric space $X$. There exists
a
neighborgood $W$ of $A$ in $X$, a  proper asymptotically Lipschitz map $g\colon
W\to \OO L$ with the following property: there exist constants $\lambda,s>0$
such
that $\alpha_L(g(a))\le \lambda d(a, X\setminus W)+s$.
\end{prop}
\begin{proof} We may assume that $L$ is a subset of $I^n$ for some $n$ and there
exists a Lipschitz retraction $r\colon U\to L$ of a neighborhood $U$ of $L$ in
$I^n$. Since $\OO I^n$ is Lipschitz equivalent to $\R^{n+1}_+$, there exists a
$(\lambda',s')$-Lipschitz extension $\tilde g\colon X\to \OO I^n$ of $g$.

Put $W=\tilde g^{-1}(\OO U)$ and $\bar g=\tilde g|W$. For every $a\in A$ and
$w\in
X\setminus W$ we have $$d(g(a),\tilde g(w)\le \lambda' d(a,w)+s'\le
\lambda'd(a,X\setminus W)+s.$$

Suppose that $d(L,I^n\setminus U)=c>0$, then, since $\tilde g(w)\notin CU$,
\begin{align*} d(g(a),\tilde g(w)= &|\alpha_L(g(a))-\alpha_L(\tilde g(w))| \\
+&\min\{\alpha_L(g(a)),\alpha_L(\tilde g(w))\} d(\beta_L(g(a)), \beta_L(\tilde
g(w))) \\
\ge & |\alpha_L(g(a))-\alpha_L(\tilde
g(w))|+c\min\{\alpha_L(g(a)),\alpha_L(\tilde g(w))\}\\
\ge & c'\alpha_L(g(a)),
\end{align*}
where $c'=\min\{c,1\}$. Then $\alpha_L(g(a))\le \lambda d(a, X\setminus W)+s$,
where $\lambda =\lambda'/c'$, $s=s'/c'$.\end{proof}

\section{Auxiliary results}

In this section we shall
collect some results needed for the proof of the
main result. They are proved in \cite{D} but it turns out that we
have also to cover the case of functions with infinite values.

A map $f\colon X\to\R_+\cup\{\infty\}$ is said to be {\em coarsely proper\/} if
the preimage $f^{-1}([0,c])$ is bounded for every $c\in \R_+$.
\begin{lemma}\label{l:1} For any function $\varphi\colon X\to  \R_+$ with
$\varphi(x)\to0$
as $x\to\infty$ the function $1/\varphi\colon X\to  \R_+\cup\{\infty\}$ is
coarsely proper.
\end{lemma}
\begin{prop}\label{p:1} Let $f\colon X\to\R_+\cup\{\infty\}$ be a coarsely
proper function. There exists an asymptotically Lipschitz proper function
$q\colon X\to\R_+$ with $q\le f$.
\end{prop}
\begin{proof} This was proved in \cite{D} for the case of $f\colon X\to\R_+$
(see Proposition 3.5). That proof also works in our case.
\end{proof}
\begin{prop}\label{p:2} Let $f_n\colon X\to\R_+\cup\{\infty\}$ be a sequence of
coarsely
proper functions. Then there exists a filtration $X=\cup_{n=1}^\infty A_n$ and a
coarsely proper function $f\colon X\to \R_+$ with $f|A_n\le n$ and
$f|(X\setminus A_n)\le f_n$ for every $n$.
\end{prop}
\begin{proof} Let $B_n=\cup_{i=1}^nf^{-1}_i([0,n])$. the sets $B_i$ are bounded
and $B_1\subset B_2\subset\dots$. Therefore, there exist bounded subsets
$A_1\subset A_2\subset\dots$ such that $A_n\cap(\cup_{i=1}^\infty B_i)=B_n$ and
$\cup_{i=1}^\infty A_i=X$. For $x\in A_n\setminus A_{n-1}$, put $f(x)=n$.
Obviously, $f$ is coarsely proper and $f|A_n\le n$. Now suppose that $x\notin
A_n$, then  $x\notin B_n$ and therefore $x\notin f^{-1}_n([0,n])$, i.e.
$f_n(x)>n\ge f|(X\setminus A_n)$.
\end{proof}

The following is an easy modification of Lemma 3.6 from \cite{D} and the proof
of it works in our case as well.
\begin{lemma} Suppose that $f\colon A\to \R_+\cup\{\infty\}$ is a coarsely
proper map defined on a closed subset $A$ of a proper metric space $X$ and
$g\colon W\to \R_+$ is a proper asymptotically Lipschitz map such that $g\le
f|W$ and  there exist $\lambda,s$ such that $\lambda d(a,X\setminus W)+s\ge
g(a)$ for every $a\in A$. Then there exists a proper  asymptotically Lipschitz
map $\bar g\colon X\to \R_+$ for which $\bar g\le f$ and $\bar g|A=g$.
\end{lemma}

\subsection{Almost geodesic spaces}
A metric space $X$ is said to be {\em almost geodesic\/} if there exists $C>0$
such that  for every two points $x,y\in X$ there is a short map  $f\colon
[0,Cd(x,y)]\to X$
with $f(0)=x$, $f(Cd(x,y))=y$. If in this definition $C=1$, then we come to the
well-known notion of {\em geodesic} space.

We are going to describe a construction of embedding of a discrete metric space
$X$ into an almost geodesic space of the  asymptotic dimension $\min\{\asd
X,1\}$.
For an unbounded  discrete metric space $X$ with base point
$x_0$  define a function $f\colon X\to[0,\infty)$ by the formula
$f(x)=d(x,x_0)$. Choose a sequence $0=t_0<t_1<t_2<\dots$ in $f(X)$ so that
$t_{i+1}>2t_i$ for every $i$. To every pair of points $x,y\in
f^{-1}([t_i,t_{i+1}])$, for some $i$, attach the line segment $[0,d(x,y)]$ along
its endpoints. Let $\hat X$ is the union of $X$ and all attached segments. We
endow $\hat X$ with the
maximal metric that agrees with the initial metric on $X$ and the standard
metric on every attached segment.

Note that since $X$ is discrete and proper, every set $f^{-1}([t_i,t_{i+1}])$ is
finite and therefore  $\hat X$ is a proper metric space.

\begin{prop} The space $\hat X$ is almost geodesic.
\end{prop}
\begin{proof}
Suppose that $x,y\in
\hat X$, then
$x\in [x_1,x_2]$, $y\in [y_1,y_2]$, where $x_1,x_2,y_1,y_2\in X$ and
$[x_1,x_2]$, $[y_1,y_2]$ are attached segments. We may suppose that
$d(x,y)=d(x,x_1)+d(x_1,y_1)+d(y_1,y)$.

Case 1): There exists $i$ such that $x_1,y_1\in f^{-1}([t_i,t_{i+1}])$. Then
$[x,x_1]\cup[x_1,y_1]\cup[y_1,y]$ is a segment of diameter $d(x,y)$ that
connects $x$ and $y$ in $\hat X$.

Case 2): $f(x_1)\in[t_i,t_{i+1}]$, $f(y_1)\in[t_j,t_{j+1}]$, where $i\neq j$.
Without loss of generality, we may assume that $i<j$.

Obviously, $d(x_1,y_1)\le d(x,y)$. Since $|t_{j}-t_{j-1}|\le d(x_1,y_1)$, we see
that  $|t_{j}-t_{j-1}|\le d(x,y)$. This implies that $t_j/2\le d(x,y)$, or
equivalently, $t_j\le d(x,y)$.
Besides, $d(y_1, f^{-1}([0,t_{j-1}])))\le d(x_1,y_1)\le d(a,b)$.

For every $k=i,i+1,\dots,j_1$ choose $z_k\in f^{-1}(t_k)$. Then
\begin{align*} d(y_1,z_{j-1})\le &d(y_1,f^{-1}([0,t_{j-1}]) +\diam
(f^{-1}([0,t_{j-1}]))\\ \le &
d(a,b)+2t_{j-1}\le d(a,b)+t_{j}\le 3d(a,b).
\end{align*}
We connect $x$ and $y$ by the segment
$$J=[x,x_1]\cup[x_1,z_1]\cup\cup_{k=i}^{j-1}[z_k,z_{k+1}])
\cup[z_{j-1},y_1]\cup[y_1,y].$$
Then
\begin{align*} \diam J\le &
d(x,x_1)+d(x_1,z_{i+1})+\left(\sum_{k=i+1}^{j-1}d(z_k,z_{k+1})\right)+
d(z_{j-1},y_1)+d(y_1,y)\\=& d(x,y)+2t_{i+1}+
\sum_{k=i+1}^{j-1}2t_{k+1}+5d(x,y)+d(x,y)\\ \le& 7d(x,y)+2(t_{i+1}+\dots+t_j)\le
7d(x,y)+4t_j\\ \le&15d(x,y).
\end{align*}
\end{proof}
We need a version of the fact proved in \cite{D} for geodesic spaces.

\begin{prop}\label{p:geod} Let $f\colon X\to Y$ be a coarse uniform map of an
almost geodesic
space
$X$. Then $f$ is asymptotically Lipschits.
\end{prop}
\begin{proof} Let $C$ be a constant from the definition of almost geodesic
space. Suppose $x,y\in X$, then there exists a short map $\alpha\colon [0,
Cd(x,y)]\to X$ such that $\alpha(0)=x$, $\alpha(Cd(x,y))=y$. There exist points
$0=t_0<t_1<\dots<t_{k-1}<t_k=Cd(x,y)$, where $k\le[d(x,y)]+1$, such that
$|t_i-t_{i-1}|\le C$ for every $i=1,\dots,k$.

Since $f$ is coarse uniform, there exists $R>0$ such that $d(f(x'),f(y')<R$
whenever $d(x',y')\le C$. Then
\begin{align*}d(f(x),f(y))\le&\sum_{i=1}^kd(f(\alpha(t_i)),
f(\alpha(t_{i-1})))\le kR\le ([d(x,y)]+1)R\\ \le& Rd(x,y)+2R.
\end{align*}
\end{proof}

\section{Asymptotic extension dimension}

Let $P$ be an object of the category $\A$. For any object $X$ of
$\A$ the {\em Kuratowski notation\/} $X\tau P$ means the following:
for every proper asymptotically Lipschitz map $f\colon A\to P$
defined on a closed subset $A$ of $X$ there is a proper
asymptotically Lipschitz extension of $f$ onto $X$.

Denote by $\LL$ the class of compact absolute Lipschitz neighborhood euclidean
extensors (ALNER).
Following \cite{}, we define a preorder relation $\leq$ on $\LL$. For
$L_1,L_2\in\LL$, we have $L_1\leq L_2$ if and only if $X\tau \OO L_1$ implies
$X\tau \OO L_2$ for all proper metric spaces $X$. This preorder relation leads
to the following equivalence relation $\sim $ on $\LL$: $L_1\sim L_2$ if and
only if $L_1\leq L_2$ and $L_2\leq L_1$. We denote by $[L]$ the equivalence
class containing $L\in\LL$.

For a proper metric space $X$, we say that its {\em asymptotic
extension dimension does not exceed $[\OO L]$} (briefly $\aed X\le
[\OO L]$ whenever $X\tau\OO L$.
If $\aed X\le [\OO L]$, then the equality $\aed X= [\OO L]$ means
the following. If we also have $\aed X\le [\OO L']$, then  $[\OO
L] \le[\OO L']$.
By  \cite{D} (see also \cite{S}), the element $[*]$
is maximal.

\begin{thm} Let $L$ be a compact metric ALNER. The following conditions are
equivalent:
\begin{enumerate}
\item $\aed X\le[\OO L]$;
\item $\ed \nu X\le[L]$.
\end{enumerate}
\end{thm}
\begin{proof} Assume that $\aed X\le[\OO L]$. Let $\varphi\colon C\to L$ be a
map
defined on a closed subset $C$ of $\nu X$. Since $L\in \mathrm{ANE}$, there
exists an
extension $\varphi'\colon V\to L$ of $\varphi$ over a closed neighborhood $V$ of
$C$ in $\bar X=X\cup\nu X$. Then $\Var_R\varphi'(x)\to0$ as $x\to\infty$, for
any fixed $R>0$. By Lemma \ref{l:1}, the function $$f_n\colon V\cap
X\to\R_+\cup\{\infty\},\
f_n(x)=\frac{1}{\Var_R\varphi'(x)},$$
is coarsely proper, for every $n\in\N$. By Proposition \ref{p:2}, there is a
coarsely proper function $f\colon V\cap X\to\R_+$ and a filtration $V\cap
X=\cup_{n=1}^\infty A_n$ such that $f|A_n\le n$ and $f|(X\setminus A_n)\le f_n$.
By Proposition 3.5 from \cite{D}, there is an asymptotically Lipschitz function
$q\colon
V\cap X\to\R_+$ with $q\le f$. We suppose that $q$ is $(\lambda,s)$-Lipschitz
for some $\lambda,s>0$. Define the map $g\colon V\cap X\to \OO L$ by the
formula $g(x)=[\varphi'(x),q(x)]$.

We are going to check that the map $g(x)$ is asymptotically Lipschitz. Let
$x,y\in V\cap X$ and $n-1\le d(x,y)\le n$.

Suppose that $x,y\in(V\cap X)\setminus A_n$, then $q(x)\le f_n(x)$, $q(y)\le
f_n(y)$. We have
\begin{align*}
d(g(x),g(y))=&|q(x)-q(y)|+\min\{q(x),q(y)\}d(\varphi'(x),\varphi'(y))\\
\le& \lambda d(x,y)+s+\min\{q(x),q(y)\}\Var_n\varphi'(x)\\
\le& \lambda d(x,y)+s+1.
\end{align*}

If $x\in A_n$, then $q(x)\le n$ and we obtain
\begin{align*} d(g(x),g(y))\le& \lambda d(x,y)+s+nd(\varphi'(x),\varphi'(y))\\
\le& \lambda d(x,y)+s+n\diam L\le \lambda d(x,y)+s+(d(x,y)+1)\diam L\\
\le&(\lambda+\diam L)d(x,y)+(s+\diam L).
\end{align*}
We argue similarly if $y\in A_n$.
Now, by the assumption, there is an asymptotically Lipschitz extension $\bar
g\colon X\to \OO L$ of $g$. Consider the composition $\eta_L\nu \bar g\colon \nu
X\to \OO L$. Obviously,  $\eta_L\nu \bar g|C=\varphi$.

Let $f\colon A\to \OO L$ be an asymptotically Lipschitz map defined on a proper
closed subset $A$ of a proper metric space $X$. By Proposition 2.2,
there is
a proper  asymptotically Lipschitz map $\tilde f\colon W\to \OO L$ and constants
$\lambda, s$ such that $\alpha_Lf(a)\le\lambda d(a,X\setminus W)+s$ for all
$a\in A$. Denote by
$\varphi\colon \nu X\to L$ an extension of the composition $\eta_L\nu \tilde f$.
Since
$L$ is an absolute neighborhood extensor, there exists an extension $\psi\colon
V\to L$ of $\varphi$ onto a closed neighborhood of $\nu X$ in the Higson
compactification $\bar X$. Extend $\psi$ to a map  $\hat\psi\colon(V\cap
X)\hat{}\to L$ as follows. Let
$J$ be a segment attached to $V$ with endpoints $a$ and $b$. We require that
$\hat\psi$ linearly maps $J$ onto a geodesic segment in $L$ with endpoints
$\psi(a)$
and $\psi(b)$.

We show that $\hat\psi $ is a slowly oscillating map. Since $\psi$ is slowly
oscillating, for every $\varepsilon>0$ and  $R>0$ there exists $K>0$ such that
$\Var_R\psi(x)<\varepsilon$ whenever $d(x,x_0)>K$. Suppose that $\hat\psi$ is
not slowly oscillating, then there exist $R>0$, $C>0$,  and sequences $(x_1^i)$,
$x_2^i$  in $(V\cap X)\hat{}$ such that $d(x_1^i,x_2^i)<R$,  $x_1^i\to\infty$,
$x_2^i\to\infty$ and $d(\hat\psi(x_1^i),\hat\psi(x_2^i))>C$ for every $i$. We
assume that $x_1^i\in[a_1^i,b_1^i]$, $x_2^i\in[a_2^i,b_2^i]$, for every $i$,
where $a_1^i,b_1^i,a_2^i,b_2^i\in X\cap V$.  Without loss of generality we may
assume that $a_1^i\to\infty$ and there exists $C_1>0$ such that
$d(\hat\psi(x_1^i),\hat\psi(a_1^i))>C_1$ for every $i$. If $d(a_1^i,b_1^i)<K$
for all $i$ and some $K>0$, then
$d(\hat\psi(x_1^i),\hat\psi(a_1^i))<d(\hat\psi(a_1^i),\hat\psi(b_1^i))\to0$, and
we obtain a contradiction. Therefore, we may assume that  $d(a_1^i,b_1^i)\to
\infty$. Then $d(a_1^i,x_1^i)/d(a_1^i,b_1^i)<R/d(a_1^i,b_1^i)\to0$ and
therefore, by the definition of the map $\hat\psi$,
$d(\hat\psi(x_1^i),\hat\psi(a_1^i))/d(\hat\psi(a_1^i),\hat\psi(b_1^i))\to0$.
Then obviously $d(\hat\psi(x_1^i),\hat\psi(a_1^i))\to0$ and we obtain a
contradiction.

Since the map $\tilde f$ is asymptotically Lipschitz, there exists
$K>0$ such that for any $a\in W$ we have $$\diam ( \alpha_L\tilde
f(N_1(a))+  \alpha_L\tilde f(a)\diam(\psi(N_1(a))\le K.$$ Define
the function $r\colon (X\cap V)\hat{}\to\R_+\cup\{\infty\}$ by the
formula $r(x)= K/(\psi(N_1(x)))$. We have $ f(a)\le r(a)$ for
every $a\in A$. The function $r$ is asymptotically proper and by
Proposition \ref{p:1}, there exists a $(\lambda',s')$-Lipschitz
function $\bar f\colon X\to\R_+$, for some $\lambda',s'$,  with
$\bar f\le r$ and $\bar f|A=\alpha_Lf$.

Define a map $g\colon (X\cap V)\hat{}\to\R$ by the
formula $g(x)=(\psi(x),\bar f(x))$. Obviously, $g|A=f$. We are going to show
that $g$ is a coarse uniform map.
Suppose $x,y\in X$, $d(x,y)<1$, then
\begin{align*}d(g(x), g(y))\le &|\bar f(x)-\bar f(y)|+\min\{\bar
f(x),\bar f(y)\}d(\psi(x),\psi(y)) \\ \le\lambda'+s'+K.\end{align*}

Note that, since $\bar f$ is proper, $g$ is also proper. Since $g$ is coarse
uniform, by Proposition \ref{p:geod}, $g$ is asymptotically Lipschitz.
\end{proof}
\begin{cor} {\bf (Finite Sum Theorem)} Suppose $X$ is a proper metric space,
$X=X_1\cup X_2$, where $X_1, X_2$ are closed subsets of $X$ with
$\aed X_i\le[\OO L]$, $i=1,2$, for some $L\in\LL$. Then $\aed
X\le[\OO L]$.
\end{cor}
\begin{proof} Since $\nu X=\nu X_1\cup \nu X_2$, the result follows from Theorem 4.1
and the finite sum theorem for extension dimension (see \cite{BD1}).
\end{proof}

\section{Remarks and open questions}

\begin{que} Does the equality $\aed \mathbb R^n=S^n$ hold?
\end{que}

\begin{que} Let $L_1,L_2$ be finite polyhedra in euclidean spaces endowed with
the induced metric. Is the inequality $[L_1]\le[L_2]$ introduced
in \cite{D2} equivalent to the inequality $[L_1]\le[L_2]$ as in
Section 4? 
\end{que}

One can define analogue of the asymptotic extension dimension
by using warped cones instead of open cones. Following \cite{R1}
we review this construction briefly. Let $\mathcal F$ be a
foliation on a compact smooth manifold $V$. Let $N$ be any
complementary subbundle to $T\mathcal F$ in $TM$.  Choose
Euclidean metrics $g_N$ in $N$ and $g_{\mathcal F}$ in $T\mathcal
F$. The {\em foliated warped cone} ${\mathcal O}_{\mathcal F}$ is
the manifold $V \times [0,\infty)/V\times\{0\}$ equipped with the
metric induced for $t \ge 1$ by the Riemannian metric $g_R +
g_{\mathcal F} + t2g_N$. The metric structure on any bounded
neighborhood of the distinguished point is irrelevant.

\begin{que} Is the obtained warped cone an absolute neighborhood extensor in the asymptotic category?
\end{que}

An affirmative answer to this question would allow us to define
asymptotic extension dimension theory with the values in warped
cones.

%\section{Acknowledgements}

\end{document}